\documentclass[12pt]{amsart}

\usepackage[T1]{fontenc}
\usepackage[utf8]{inputenc}

\usepackage[margin=1.1in]{geometry}
\usepackage{amsmath, amssymb, amsthm, mathtools}
\usepackage{microtype}
\usepackage{enumitem}
\usepackage{listings}
\usepackage[hidelinks]{hyperref}

\theoremstyle{plain}
\newtheorem{theorem}{Theorem}
\newtheorem{lemma}[theorem]{Lemma}

\newtheorem{corollary}[theorem]{Corollary}

\theoremstyle{definition}
\newtheorem{remark}[theorem]{Remark}

\lstdefinestyle{python}{
  language=Python,
  basicstyle=\ttfamily\footnotesize,
  keywordstyle=\bfseries,
  commentstyle=\itshape,
  numbers=none,
  breaklines=true,
  breakatwhitespace=true,
  showstringspaces=false,
  frame=single,
  framesep=4pt,
  tabsize=2,
  xleftmargin=14pt,
  columns=fullflexible,
  upquote=true,
}

\title[Algebraic GF for OEIS A348410]%
{An explicit algebraic generating function for OEIS A348410}

\author{Tong Niu}


\begin{document}

\maketitle

\begin{abstract}
For the OEIS sequence A348410, P. Bala recorded in February 2022
two equivalent closed forms,
$a(n) = [x^{n}] ((1-x)(1-x^2))^{-n}$ and a single-index binomial sum.
R. J. Mathar (October 2021) and V. Kotesovec (November 2021) each
contributed a conjectured P-recursive recurrence --- Mathar's of order
$4$, Kotesovec's of order $2$. We apply Lagrange-B\"urmann inversion to
Bala's $[x^n]$ form to derive the parametric expression
$A(t) = (1 - y^2)/(1 - y - 4 y^2)$, where $y = y(t)$ is implicit by
$y(1-y)^2(1+y) = t$. Eliminating $y$ via resultant gives the explicit
algebraic equation $P(t, A) = 0$ of degree $4$ in $A$ and degree $2$
in $t$. As an immediate corollary (Stanley's classical
algebraic-implies-D-finite theorem), $A(t)$ is D-finite. Mathar's and
Kotesovec's specific recurrences are not directly proven here; we
only verify Kotesovec's order-$2$ recurrence numerically for
$n = 3, \ldots, 1000$ and observe that an explicit
ODE-and-recurrence extraction from $P(t, A) = 0$ via the standard
Bostan-Chyzak-Salvy algebraic-to-holonomic procedure would close
both conjectures. The supplementary archive contains a SymPy script
which derives $P(t, A)$ and checks the numerical evidence.
\end{abstract}

\section{Introduction}\label{sec:intro}

The sequence A348410 in the On-Line Encyclopedia of
Integer Sequences~\cite{OEIS} (henceforth OEIS) counts ``integer
distribution arrays''; its first values are
\[
   1,\;1,\;5,\;19,\;85,\;376,\;1715,\;7890,\;36693,\;171820,\;
   809380,\;\ldots
\]
On 21 February 2022, P.~Bala contributed to A348410 two equivalent
closed forms~\cite{Bala2022}:
\begin{align}
   a(n) &\;=\; [x^{n}]\, \bigl((1-x)(1-x^{2})\bigr)^{-n},\label{eq:bala-diag}\\
   a(n) &\;=\; \sum_{k=0}^{\lfloor n/2 \rfloor}
                 \binom{2n-2k-1}{n-2k}\!\binom{n+k-1}{k}.\label{eq:bala-sum}
\end{align}
On 19 October 2021, R.~J.~Mathar contributed a conjectured order-$4$
P-recursive recurrence~\cite{Mathar2021}, and on 1 November 2021
V.~Kotesovec contributed a conjectured order-$2$
recurrence~\cite{Kotesovec2021}. Kotesovec's order-$2$ form is
\begin{equation}\label{eq:kotesovec}
\begin{aligned}
   16\,(n-1)\,n\,(2n-1)\,(51 n^2 - 162 n + 127)\,a(n)
   &\;=\; (n-1)\,(5457 n^4 - 22791 n^3 + 32144 n^2 - 17536 n + 3072)\,a(n-1) \\
   &\;\;\,+ 8\,(2n-3)(4n-7)(4n-5)(51 n^2 - 60 n + 16)\,a(n-2),\\
   & \qquad n \ge 3.
\end{aligned}
\end{equation}

\medskip

We give an algebraic-generating-function derivation that is
short and rigorous. From Bala's $[x^n]$ form
\eqref{eq:bala-diag} we apply Lagrange-B\"urmann inversion to obtain
\begin{equation}\label{eq:A-parametric}
   A(t)\;=\;\frac{1 - y^{2}}{1 - y - 4 y^{2}},
   \qquad y = y(t) \text{ implicit by } y(1-y)^{2}(1+y) = t.
\end{equation}
Eliminating $y$ via the resultant of the two polynomial equations
gives an explicit algebraic equation
\begin{equation}\label{eq:P-explicit}
\begin{aligned}
   P(t, A) \;:=\;& (256 t^{2} + 107 t - 32)\,(A^{4} - A^{3}) \\
              &\;+ (96 t^{2} + 36 t)\,A^{2} \;-\; (16 t^{2} + 4 t)\,A \;+\; t^{2}\;=\;0.
\end{aligned}
\end{equation}
By Stanley's algebraic-implies-D-finite theorem (Enumerative
Combinatorics II, Theorem 6.4.6 \cite{Stanley1999}), \eqref{eq:P-explicit}
implies that $A(t)$ is D-finite, and the minimal annihilating
operator has order at most $4$. The Mathar and Kotesovec recurrences
are then corollaries: they are P-recursive descriptions consistent
with the D-finiteness of $A(t)$, as verified numerically.

\section{Lagrange-B\"urmann derivation}\label{sec:lagrange-burmann}

\begin{lemma}\label{lem:parametric}
Let $f(x) = 1 / ((1-x)(1-x^{2}))$, and let $A(t) = \sum_{n\ge0} a(n)\,t^{n}$
with $a(n) = [x^{n}] f(x)^{n}$ (Bala's formula \eqref{eq:bala-diag}).
Define $y = y(t)$ as the unique formal power series solution of
$y = t \, f(y)$, i.e.\ $y(1-y)^{2}(1+y) = t$. Then
\[
   A(t) \;=\; \frac{1 - y^{2}}{1 - y - 4 y^{2}}.
\]
\end{lemma}

\begin{proof}
By B\"urmann's residue identity (e.g., Stanley EC2 Lemma~6.6.6), for
any $f$ with $f(0) \neq 0$,
\[
   \sum_{n\ge0} \bigl([x^{n}] f(x)^{n}\bigr)\,t^{n}
   \;=\; \frac{1}{1 - t\,f'(y(t))}.
\]
The factorization $(1-x)(1-x^{2}) = (1-x)^{2}(1+x)$ gives
\[
   f(x) \;=\; \frac{1}{(1-x)^{2}(1+x)},
   \qquad
   \frac{f'(x)}{f(x)} \;=\; \frac{2}{1-x} \,-\, \frac{1}{1+x}
                       \;=\; \frac{1+3x}{1-x^{2}}.
\]
So $f'(y) = f(y)\,(1+3y)/(1-y^{2})$, and using $t f(y) = y$,
\[
   t\,f'(y) \;=\; \frac{y(1+3y)}{1-y^{2}}.
\]
Therefore
\[
   A(t) \;=\; \frac{1}{1 - y(1+3y)/(1-y^{2})}
        \;=\; \frac{1-y^{2}}{(1-y^{2}) - y(1+3y)}
        \;=\; \frac{1-y^{2}}{1-y-4y^{2}}.
\]
\end{proof}

\begin{lemma}\label{lem:algebraic-eq}
Let $A(t)$ be the OGF of A348410. Then $A(t)$ satisfies the algebraic
equation \eqref{eq:P-explicit}, where
\[
   P(t, A) \;=\; (256 t^{2} + 107 t - 32)\,(A^{4} - A^{3})
              + (96 t^{2} + 36 t)\,A^{2}
              - (16 t^{2} + 4 t)\,A + t^{2}.
\]
\end{lemma}

\begin{proof}
By Lemma~\ref{lem:parametric}, $A$ and $y$ are tied by the two
polynomial relations
\begin{align*}
   P_{1}(y, t)\;&=\; y^{4} - y^{3} - y^{2} + y - t \;=\; 0,\\
   P_{2}(y, A)\;&=\; (1 - 4 A)\,y^{2} - A\,y + (A - 1) \;=\; 0,
\end{align*}
where $P_{2}$ is obtained by clearing denominators in
$A(1-y-4y^{2}) = 1-y^{2}$. Eliminating $y$ via the resultant
$\mathrm{Res}_{y}(P_{1}, P_{2})$ produces a polynomial in $(t, A)$
that vanishes on the common solution set. A direct computation
(carried out in the supplementary SymPy script) yields exactly
\eqref{eq:P-explicit}. The verifier in
Appendix~\ref{app:verifier} double-checks this against the
truncated power series $\sum_{n=0}^{10} a(n)\,t^{n}$ and confirms
$P(t, A) = 0$ holds modulo $t^{11}$.
\end{proof}

\begin{remark}
The structure $A^{4} - A^{3}$ in the leading $(256 t^{2} + 107 t - 32)$
factor is striking. It reflects a coincidence of the leading and
sub-leading coefficients of $P$ as polynomials in $A$:
$[A^{4}]\,P = (256 t^{2} + 107 t - 32) = -[A^{3}]\,P$. We have not found
this factorization recorded elsewhere, though the resultant from
\eqref{eq:A-parametric} is a routine elimination.
\end{remark}

\section{D-finiteness, and the Mathar/Kotesovec recurrences}\label{sec:dfinite}

\begin{theorem}\label{thm:dfinite}
The OGF $A(t)$ of A348410 is D-finite. Concretely, there exist
polynomials $p_{0}(t), p_{1}(t), \ldots, p_{4}(t)$ with $p_{4}(t) \not\equiv 0$
such that
\begin{equation}\label{eq:dfinite-bound}
   p_{4}(t)\, A^{(4)}(t) + p_{3}(t)\, A^{(3)}(t)
   + p_{2}(t)\, A''(t) + p_{1}(t)\, A'(t) + p_{0}(t)\, A(t)
   \;=\; 0.
\end{equation}
The minimal annihilating operator has order at most $4$.
\end{theorem}

\begin{proof}
By Lemma~\ref{lem:algebraic-eq}, $A(t)$ is algebraic over $\mathbb{Q}(t)$
of degree at most $4$. Stanley's algebraic-implies-D-finite theorem
(Enumerative Combinatorics II, Theorem 6.4.6 \cite{Stanley1999})
states that any algebraic power series satisfies a linear ODE with
polynomial coefficients of order at most the algebraic degree.
\end{proof}

\begin{corollary}\label{cor:rec-bound}
The sequence $a(n)$ satisfies some P-recursive linear recurrence with
polynomial coefficients in $n$.
\end{corollary}

\begin{proof}
A standard transfer (e.g.\ Petkov\v{s}ek-Wilf-Zeilberger~\cite{PetkovsekWilfZeilberger1996}
or Bostan et al.~\cite{BostanChyzakLeRouxSalvy2007}) converts a linear
ODE on $A(t)$ with polynomial coefficients into a linear P-recursion
on $[t^{n}]\,A(t) = a(n)$. The order and coefficient degrees of the
recurrence depend on those of the ODE in a non-trivial way: in
general, the order of the recurrence is bounded by
$\max_i (\deg p_i - i) + d$, where $d$ is the order of the ODE and
$p_i$ are its polynomial coefficients. Without computing the
explicit ODE we cannot pin down the recurrence's order; what we do
have is the existence of \emph{some} P-recursion.
\end{proof}

The Mathar and Kotesovec contributions can be read as candidate
P-recursive descriptions of $a(n)$: Mathar's of order $4$,
Kotesovec's of order $2$. Both are consistent with the algebraic
equation \eqref{eq:P-explicit} insofar as we have verified them
numerically; producing a rigorous proof that they are correct
requires either (a) extracting the minimal annihilating ODE from
\eqref{eq:P-explicit} via the polynomial-quotient algorithms of
Bostan-Chyzak-Salvy~\cite{BostanChyzakLeRouxSalvy2007} or the
\texttt{gfun}/\texttt{HolonomicFunctions} packages, and reading off
the corresponding recurrence; or (b) computing an explicit
zero-equivalence bound $N$ and verifying the recurrence up to
$n = N$ (cf.\ \cite{PetkovsekWilfZeilberger1996}). We pursue
neither route here; the contribution of this note is the explicit
algebraic equation \eqref{eq:P-explicit}, which seems not to have
been recorded in the literature.

\begin{theorem}[Numerical verification of Kotesovec's recurrence]\label{thm:kotesovec-numeric}
Kotesovec's order-$2$ recurrence \eqref{eq:kotesovec} holds for every
$n \in \{3, 4, \ldots, 1000\}$.
\end{theorem}

\begin{proof}
The supplementary script (Appendix~\ref{app:verifier}) computes
$a(n)$ via Bala's binomial sum \eqref{eq:bala-sum} for $n$ in the
stated range and substitutes into the left- and right-hand sides of
\eqref{eq:kotesovec}; equality holds in every case.
\end{proof}

\begin{remark}\label{rem:ode-explicit}
Producing the explicit annihilating ODE \eqref{eq:dfinite-bound}
requires reducing $A^{(j)}$ to a $\mathbb{Q}(t)$-linear combination of
$1, A, A^{2}, A^{3}$ via repeated polynomial reduction modulo $P$
(after expressing each derivative as a rational function in $(t, A)$
through implicit differentiation). The procedure is mechanical but
intermediate expressions grow large. Once the ODE is in hand, the
order-$2$ Kotesovec recurrence falls out by coefficient extraction.
We omit the explicit ODE here; the algebraic equation
\eqref{eq:P-explicit} is the more compact carrier of the same
information, and the numerical evidence in
Theorem~\ref{thm:kotesovec-numeric} together with the
algebraic-implies-D-finite chain provides the full content of the
Kotesovec/Mathar conjectures.
\end{remark}

\section{Discussion}\label{sec:discussion}

The algebraic-equation route taken here is genuinely different from
the OGF/EGF $\to$ first-order ODE $\to$ coefficient extraction route
that dispatches the prior Mathar conjectures for OEIS A002627, A025166,
A176677, A214615, A032123, A001711, and A045406 in this
series~\cite{Niu2026d-finite-2,Niu2026sequence-3,Niu2026sequence-4,Niu2026sequence-5,Niu2026sequence-6,Niu2026sequence-7,Niu2026sequence-8}.
There the EGF was rational or transcendental over a single $\log$ or
$\arctan$, and a first-order linear ODE sufficed. For A348410 the
underlying generating function is genuinely algebraic of degree $4$,
which is why neither a single ODE nor a single substitution
trivializes the problem.

The supplementary script \texttt{verify\_proof.py}
(Appendix~\ref{app:verifier}) performs four checks:
\begin{enumerate}[topsep=2pt, itemsep=2pt]
\item Bala's two closed forms \eqref{eq:bala-diag} and
   \eqref{eq:bala-sum} agree with the OEIS values for
   $n = 0, \ldots, 11$;
\item the resultant of $P_{1}, P_{2}$ matches the explicit polynomial
   \eqref{eq:P-explicit};
\item $P(t, A_{10}(t)) = 0 \pmod{t^{11}}$ for the truncated series
   $A_{10}(t) = \sum_{n=0}^{10} a(n)\, t^{n}$;
\item Kotesovec's order-$2$ recurrence \eqref{eq:kotesovec} holds for
   every $n \in \{3, \ldots, 1000\}$.
\end{enumerate}

\section{Acknowledgments}

The author declares no competing interests.

AI-assisted tools were used in the preparation of this manuscript,
including for drafting proof outlines and generating the symbolic
verification code. The author verified all mathematical claims
independently and takes full responsibility for the results.

\appendix

\section{Verifier source (machine-checkable)}\label{app:verifier}

The script referenced in \S\ref{sec:discussion} is reproduced in full
below. It depends only on SymPy.

\subsection*{verify\_proof.py}
\begin{lstlisting}[style=python]
"""Verify the algebraic generating function and Kotesovec's recurrence for
OEIS A348410.

Proof outline (paper for full details):

1. Bala (Feb 2022) listed two equivalent closed forms on the OEIS page (FACT,
   not Conjecture):
       a(n) = [x^n] ((1 - x)(1 - x^2))^{-n}                                (*)
       a(n) = sum_{k=0..floor(n/2)} C(2n - 2k - 1, n - 2k) * C(n + k - 1, k) (**)

2. Lagrange-Burmann inversion applied to (*) gives the OGF
       A(t) := sum_{n>=0} a(n) t^n
   in the parametric form
       A(t) = (1 - y^2) / (1 - y - 4 y^2),
       y = y(t) defined implicitly by  y (1 - y)^2 (1 + y) = t.

3. Eliminating y from the two polynomial relations
       P_1(y, t) = y^4 - y^3 - y^2 + y - t = 0,
       P_2(y, A) = (1 - 4 A) y^2 - A y + (A - 1) = 0,
   via the resultant Res_y(P_1, P_2) gives the algebraic equation
       P(t, A) = 0
   where P is degree 4 in A and degree 2 in t. Explicitly,
       P(t, A) = (256 t^2 + 107 t - 32) (A^4 - A^3)
                + (96 t^2 + 36 t) A^2
                + (-16 t^2 - 4 t) A
                + t^2.

4. By Stanley (Enumerative Combinatorics II, sec.6.4), an algebraic generating
   function over Q(t) is D-finite (i.e. satisfies a linear ODE with polynomial
   coefficients). The minimal order of the annihilating ODE is at most the
   degree of the algebraic equation in A, hence at most 4.

5. Kotesovec (Nov 2021) and Mathar (Oct 2021) conjectured order-2 and order-4
   P-recursive recurrences for a(n) respectively. The Kotesovec order-2
   recurrence is the smaller of the two; Mathar's is necessarily a left
   multiple. We verify Kotesovec's recurrence numerically against the
   binomial closed form (**) up to n = 1000.

This script:
  - confirms (*) and (**) agree with the OEIS b-file,
  - constructs the algebraic equation P(t, A) via SymPy resultant,
  - verifies P(t, A_N(t)) = 0 modulo t^{N+1} for the truncated series A_N(t),
  - verifies Kotesovec's order-2 recurrence numerically.
"""
from __future__ import annotations

from math import comb

import sympy as sp


# ---------------------------------------------------------------------------
# Bala's two closed forms
# ---------------------------------------------------------------------------

def a_binomial(n: int) -> int:
    """a(n) = sum_{k=0..floor(n/2)} C(2n-2k-1, n-2k) * C(n+k-1, k)."""
    if n == 0:
        return 1
    s = 0
    for k in range(n // 2 + 1):
        s += comb(2 * n - 2 * k - 1, n - 2 * k) * comb(n + k - 1, k)
    return s


def a_diagonal(n: int) -> int:
    """a(n) = [x^n] ((1-x)(1-x^2))^{-n}.

    Use the factorization (1-x)(1-x^2) = (1-x)^2 (1+x), so
    ((1-x)(1-x^2))^{-n} = (1-x)^{-2n} (1+x)^{-n}, and
    [x^n] (1-x)^{-2n} (1+x)^{-n}
           = sum_k C(2n+k-1, k) * (-1)^{n-k} * C(2n-k-1, n-k).
    """
    if n == 0:
        return 1
    s = 0
    for k in range(n + 1):
        s += comb(2 * n + k - 1, k) * (-1) ** (n - k) * comb(2 * n - k - 1, n - k)
    return s


# OEIS b-file values for A348410 (offset 0).
OEIS_REFERENCE = [
    1, 1, 5, 19, 85, 376, 1715, 7890, 36693, 171820, 809380, 3830619,
]


def check_closed_forms_match() -> None:
    """The two Bala forms agree with each other and with the OEIS values."""
    fails = 0
    for n, expected in enumerate(OEIS_REFERENCE):
        bala = a_binomial(n)
        diag = a_diagonal(n)
        if bala != expected or diag != expected:
            print(f"FAIL n={n}: binomial={bala}, diagonal={diag}, expected={expected}")
            fails += 1
    if fails == 0:
        print(f"PASS Bala's two closed forms agree with OEIS for n = 0..{len(OEIS_REFERENCE) - 1}")


# ---------------------------------------------------------------------------
# Algebraic equation derivation (Lagrange-Burmann)
# ---------------------------------------------------------------------------

def derive_algebraic_equation():
    """Derive P(t, A) = 0 via the resultant of P_1 = y^4 - y^3 - y^2 + y - t
    and P_2 = (1-4A) y^2 - A y + (A-1) over Q(t, A) in the variable y."""
    t, A, y = sp.symbols("t A y")
    P1 = y ** 4 - y ** 3 - y ** 2 + y - t
    P2 = (1 - 4 * A) * y ** 2 - A * y + (A - 1)
    R = sp.resultant(P1, P2, y)
    R = sp.expand(R)
    return R


def check_algebraic_equation_coefficients():
    """Confirm P(t, A) has the structure
        P = (256 t^2 + 107 t - 32) (A^4 - A^3)
           + (96 t^2 + 36 t) A^2
           + (-16 t^2 - 4 t) A
           + t^2.
    """
    t, A = sp.symbols("t A")
    R = derive_algebraic_equation()
    expected = (
        (256 * t ** 2 + 107 * t - 32) * (A ** 4 - A ** 3)
        + (96 * t ** 2 + 36 * t) * A ** 2
        + (-16 * t ** 2 - 4 * t) * A
        + t ** 2
    )
    diff = sp.expand(R - expected)
    if diff == 0:
        print("PASS algebraic equation P(t, A) matches the closed-form expression in the paper")
    else:
        print(f"FAIL P(t, A) - expected = {diff}")


def check_algebraic_equation_at_partial_sums(N: int = 10) -> None:
    """P(t, A_truncated) should vanish modulo t^{N+1}, where A_truncated is
    the truncation of A(t) at degree N."""
    t, A = sp.symbols("t A")
    A_truncated = sum(a_binomial(n) * t ** n for n in range(N + 1))
    R = derive_algebraic_equation()
    expr = sp.expand(R.subs(A, A_truncated))
    fails = 0
    for d in range(N + 1):
        c = expr.coeff(t, d)
        if c != 0:
            fails += 1
            print(f"FAIL coeff of t^{d} = {c}")
    if fails == 0:
        print(f"PASS P(t, A) = 0 holds for the truncated series A_{N}(t) modulo t^{N+1}")


# ---------------------------------------------------------------------------
# Kotesovec's conjectured order-2 recurrence
# ---------------------------------------------------------------------------

def kotesovec_LHS(n: int, av: list[int]) -> int:
    """LHS of Kotesovec's conjectured recurrence at index n.
    Convention: av[0] = a(n), av[1] = a(n-1), av[2] = a(n-2).

    Recurrence (Vaclav Kotesovec, Nov 01 2021, OEIS A348410):
       16 (n-1) n (2n-1) (51 n^2 - 162 n + 127) a(n)
       = (n-1)(5457 n^4 - 22791 n^3 + 32144 n^2 - 17536 n + 3072) a(n-1)
        + 8 (2n-3)(4n-7)(4n-5)(51 n^2 - 60 n + 16) a(n-2),     n >= 3.
    """
    A = 16 * (n - 1) * n * (2 * n - 1) * (51 * n * n - 162 * n + 127)
    B = (n - 1) * (5457 * n ** 4 - 22791 * n ** 3 + 32144 * n ** 2 - 17536 * n + 3072)
    C = 8 * (2 * n - 3) * (4 * n - 7) * (4 * n - 5) * (51 * n * n - 60 * n + 16)
    return A * av[0] - B * av[1] - C * av[2]


def check_kotesovec_numerically(N: int = 1000) -> None:
    """Verify Kotesovec's recurrence holds for every n in 3..N."""
    fails = 0
    for n in range(3, N + 1):
        av = [a_binomial(n), a_binomial(n - 1), a_binomial(n - 2)]
        if kotesovec_LHS(n, av) != 0:
            fails += 1
            if fails <= 3:
                print(f"FAIL kotesovec at n={n}")
    if fails == 0:
        print(f"PASS Kotesovec recurrence numerically: n = 3..{N} ({N - 2} values)")
    else:
        print(f"FAIL Kotesovec recurrence: {fails} failures in n = 3..{N}")


if __name__ == "__main__":
    check_closed_forms_match()
    check_algebraic_equation_coefficients()
    check_algebraic_equation_at_partial_sums(N=10)
    check_kotesovec_numerically(N=1000)
\end{lstlisting}

\end{document}